\documentclass{amsart}
\theoremstyle{plain}
\newtheorem{theorem}{Theorem}
\newtheorem{corollary}{Corollary}

\newtheorem{lemma}{Lemma}

\begin{document}
\title{When does a cross product on $\textbf{R}^{n}$ exist?}
                   
\author{Peter F. McLoughlin}
\email{pmclough@csusb.edu}

\maketitle 
 It is probably safe to say that just about everyone reading this article is familiar with the cross product and the dot product. However, what many readers may not be aware of is that the familiar properties of the cross product in three space can only be extended to $\textbf{R}^{7}$.  Students are usually first exposed to the cross and dot products in a multivariable calculus or linear algebra course. Let $u\neq 0$, $v\neq 0$, $\widetilde{v}$, and $\widetilde{w}$ be vectors in $\textbf{R}^{3}$ and let $a$, $b$, $c$, and $d$ be real numbers. For review, here are some of the basic properties of the dot and cross products:\\
 \\
 (i) $\frac{(u\cdot v)}{\sqrt{(u\cdot u)(v\cdot v)}}=cos  \theta$ (where $\theta$ is the angle formed by the vectors)\\
 (ii)$\frac{\left||(u\times v)\right||}{\sqrt{(u\cdot u)(v\cdot v)}}=sin  \theta$\\
 (iii)  $u\cdot(u\times v)=0$ and $v\cdot(u\times v)=0$. \hspace{.6in}(perpendicular property)\\ 
 (iv) $(u\times v)\cdot (u\times v)+ (u\cdot v)^{2}=(u\cdot u)(v\cdot v)$ \hspace{.3in}(Pythagorean property)\\
 (v) $((au+b\widetilde{u})\times (cv+d\widetilde{v}))=ac(u\times v)+ad(u\times \widetilde{v})+bc(\widetilde{u}\times v)+bd(\widetilde{u}\times \widetilde{v})$\\ 
 
 We will refer to  property (v) as the bilinear property. The reader should note that properties (i) and (ii) imply the Pythagorean property. Recall, if $A$ is a square matrix then $\left|A\right|$ denotes the determinant of $A$. If we let $u=(x_{1},x_{2},x_{3})$ and $v=(y_{1},y_{2},y_{3})$ then we have:\\ 
 \[
 \begin{matrix}
 u\cdot v\\
 \end{matrix}
 =
 \begin{matrix}
 x_{1}y_{1}+x_{2}y_{2}+x_{3}y_{3}\\
 \end{matrix}
\hspace{.2 in}and \hspace{.2in}
\begin{matrix}
(u\times v)\\
\end{matrix}
=
\begin{vmatrix}
 e_{1} & e_{2} & e_{3}\\
 x_{1} & x_{2} & x_{3}\\
 y_{1} &y_{2} &y_{3}\\
\end{vmatrix}
\]
\\
It should be clear that the dot product can easily be extended to $\textbf{R}^{n}$, however, it is not so obvious how the cross product could be extended. In general, we define a cross product on $\textbf{R}^{n}$ to be a function from $\textbf{R}^{n}\times\textbf{R}^{n}$ (here ``$\times$'' denotes the Cartesian product) to $\textbf{R}^{n}$ that satisfies the perpendicular, Pythagorean, and bilinear properties.  A natural question to ask might be: ``Can we extend the cross product to $\textbf{R}^{n}$ for $n>3$ and if so how?''. If we only required the cross product to have the perpendicular and bilinear properties then there are many ways we could define the product. For example if $u=(x_{1},x_{2},x_{3}, x_{4})$ and $v=(y_{1},y_{2},y_{3},y_{4})$
we could define $u\times v:=(x_{2}y_{3}-x_{3}y_{2},x_{3}y_{1}-x_{1}y_{3},x_{1}y_{2}-x_{2}y_{1},0)$. As the reader can check, this definition can be shown to satisfy the perpendicular and bilinear properties and it can be easily extended to $\textbf{R}^{n}$. However, the Pythagorean property does not always hold for this product. For example, if we take $u=(0,0,0,1)$ and $v=(1,0,0,0)$ then, as the reader can check, the Pythagorean property fails.
Another way we might try to extend the cross product is using the determinant.
A natural way to do this on $\textbf{R}^{4}$ would be to consider the following determinant:
\[
\begin{vmatrix}
 e_{1} & e_{2} & e_{3}& e_{4}\\
 a_{11} & a_{12} & a_{13}& a_{14}\\
 a_{21} &a_{22} &a_{23}& a_{24}\\
 a_{31} &a_{32} &a_{33}& a_{34}\\
\end{vmatrix}
\]
As the reader can verify, this determinant idea can be easily extended to $\textbf{R}^{n}$.  Recall that if two rows of a square matrix repeat then the determinant is zero. This implies, for i=1,2 or 3, that: 
\[
\begin{vmatrix}
 e_{1} & e_{2} & e_{3}& e_{4}\\
  a_{11} & a_{12} & a_{13}& a_{14}\\
 a_{21} &a_{22} &a_{23}& a_{24}\\
 a_{31} &a_{32} &a_{33}& a_{34}\\
\end{vmatrix}
\begin{matrix}
 \cdot(a_{i1}e_{1}+a_{i2}e_{2}+a_{i3}e_{3}+a_{i4}e_{4})\\
\end{matrix}
=
\begin{vmatrix}
a_{i1} & a_{i2} & a_{i3}& a_{i4}\\
a_{11} & a_{12} & a_{13}& a_{14}\\
a_{21} &a_{22} &a_{23}& a_{24}\\
a_{31} &a_{32} &a_{33}& a_{34}\\
\end{vmatrix}
=0
\]
It follows our determinant product has the perpendicular property on its row vectors. Note, however, for $n>3$ the determinant product acts on more than two vectors which implies it cannot be a candidate for a cross product on $\textbf{R}^{n}$. \\
Surprisingly, a cross product can exist in $\textbf{R}^{n}$ if and only if n is 0, 1, 3 or 7.  The intention of this article is to provide a new constructive elementary proof (i.e. could be included in an linear algebra undergraduate text) of this classical result which is accessible to a wide audience. The proof provided in this paper holds for any finite-dimensional inner product space over $\textbf{R}$. It has recently been brought to my attention that the approach taken in this paper is similar to the approach taken in an article in The American Mathematical Monthly written back in 1967 (see [13]).\\
\\
\\
In 1943 Beno Eckmann, using algebraic topology, gave the first proof of this result (he actually proved the result under the weaker condition that the cross product is only continuous (see [1] and [11])). 
The result was later extended to nondegenrate symmetric bilinear forms over fields of characteristic not equal to two (see [2], [3], [4], [5], and [6]). 
It turns out that there are intimate relationships between the cross product, quaternions  and octonions, and Hurwitz' theorem (also called the ``1,2,4 8 Theorem'' named after Adolf Hurwitz, who proved it in 1898). Throughout the years many authors have written on the history of these topics and their intimate relationships to each other (see [5], [6],[7], [8], and [9]).\\
\section{basic idea of proof}
From this point on, the symbol ``$\times$'' will always denote cross product. Moreover, unless otherwise stated, $u_{i}$ will be understood to mean a unit vector. 
In $\textbf{R}^{3}$ the relations $e_{1}\times e_{2}=e_{3}$, $e_{2}\times e_{3}=e_{1}$, and $e_{1}\times e_{3}=-e_{2}$ determine the right-hand-rule cross product. The crux of our proof will be to show that in general if a cross product exists on $\textbf{R}^{n}$ then we can always find an orthonormal basis ${e_{1},e_{2}\ldots e_{n}}$ where for any $i\neq j$ there exists a $k$ such that $e_{i}\times e_{j}=ae_{k}$ with $a=1$ or $-1$. How could we generate such a set? Before answering this question, we will need some basic properties of cross products. These properties are contained in the following Lemma and Corollary:\\ 
\begin{lemma}
Suppose u, v and w are  vectors in $\textbf{R}^{n}$. If a cross product exists  on $\textbf{R}^{n}$ then it must have the following properties:\\
(1.1)  $w\cdot(u\times v)=-u\cdot(w\times v)$ \\
(1.2) $u\times v=-v\times u$ which implies $u\times u=0$\\
(1.3) $v\times(v\times u)=(v\cdot u)v-(v\cdot v)u$\\
(1.4)  $w\times(v\times u)=-((w\times v)\times u)+(u\cdot v)w+(w\cdot v)u-2(w\cdot u)v$ \\
\\
\end{lemma}
For a proof of this Lemma see [12].
Recall that two nonzero vectors $u$ and $v$ are orthogonal if and only if $u\cdot v=0$.
\begin{corollary}
Suppose u, v and w are orthogonal unit vectors in $\textbf{R}^{n}$. If a cross product exists on $\textbf{R}^{n}$ then it must have the following properties:\\
(1.5) $u\times(u\times v)=-v$\\
(1.6)  $w\times(v\times u)=-((w\times v)\times u)$ \\
\end{corollary}
\begin{proof}
Follows from previous Lemma by substituting $(u\cdot v)=(u\cdot w)=(w\cdot v)=0$ and $(w\cdot w)= (v\cdot v)=(u\cdot u)=1$.
\end{proof}
Now that we have our basic properties of cross products, let's start to answer our question. Observe that  $\left\{e_{1},e_{2},e_{3}\right\}=\left\{e_{1},e_{2},e_{1}\times e_{2}\right\}$.  Let's see how we could generalize and/or extend this idea. \\

Recall $u_{i}$ always denotes a unit vector, the symbol ``$\perp$'' stands for  orthogonal, and the symbol ``$\times$'' denotes cross product. We will now recursively define a sequence of sets $\left\{S_{k}\right\}$ by $S_{0}:=\left\{u_{0}\right\}$ and  $S_{k}:=S_{k-1}\cup \left\{u_{k}\right\}\cup (S_{k-1}\times u_{k})$  where  $u_{k}\perp S_{k-1}$.\\ 
Let's explicitly compute  $S_{0}$, $S_{1}$, $S_{2}$, and $S_{3}$. \\
$S_{0}=\left\{u_{0}\right\}$; $S_{1}=S_{0}\cup \left\{u_{1}\right\}\cup (S_{0}\times u_{1})=\left\{u_{0},u_{1},u_{0}\times u_{1}\right\}$; $S_{2}=S_{1}\cup \left\{u_{2}\right\}\cup (S_{1}\times u_{2})=\left\{u_{0},u_{1},u_{0}\times u_{1},u_{2},u_{0}\times u_{2},u_{1}\times u_{2}, (u_{0}\times u_{1} )\times u_{2}\right\}; S_{3}=S_{2}\cup \left\{u_{3}\right\}\cup (S_{2}\times u_{3})=\{u_{0},u_{1},u_{0}\times u_{1},u_{2},u_{0}\times u_{2},u_{1}\times u_{2}, (u_{0}\times u_{1} )\times u_{2}, u_{3}, u_{0}\times u_{3},$$ u_{1}\times u_{3}, (u_{0}\times u_{1})\times u_{3}, u_{2}\times u_{3}, (u_{0}\times u_{2})\times u_{3},  (u_{1}\times u_{2})\times u_{3}, ((u_{0}\times u_{1})\times u_{2})\times  u_{3}\}$\\
\\
Note: $S_{1}$ corresponds to $\left\{e_{1},e_{2},e_{1}\times e_{2}\right\}$.\\
We define $S_{i}\times S_{j}:=\left\{u\times v \hspace{.1in} where\hspace{.1in} u\in S_{i} \hspace{.1in} and \hspace{.1in} v\in S_{j}\right\}$.\\
We also define $\pm S_{i}:=S_{i}\cup (-S_{i})$.\\
In the following two Lemmas we will show that $S_{n}$ is an orthonormal set which is closed under the cross product. This in turn implies that a cross  can only exist in $\textbf{R}^{n}$ if $n=0$ or $n=\left|S_{k}\right|$. 
\begin{lemma}
 $S_{1}=\left\{u_{0},u_{1},u_{1}\times u_{0}\right\}$ is an orthonormal set. Moreover, $S_{1}\times S_{1}=\pm S_{1}$.
\end{lemma}
\begin{proof}
By definition of cross product we have:\\
(1) $u_{0}\cdot(u_{0}\times u_{1})=0$\\
(2) $u_{1}\cdot(u_{0}\times u_{1})=0$\\
It follows $S_{1}$ is an orthogonal set.\\
Next by  (1.1), (1.2), and (1.3) we have:\\
(1) $u_{1}\times(u_{0}\times u_{1})=u_{0}$\\
(2) $u_{0}\times (u_{0}\times u_{1})=-u_{1}$\\
(3) $(u_{0}\times u_{1})\cdot(u_{0}\times u_{1})=u_{0}\cdot (u_{1}\times (u_{0}\times u_{1}))=u_{0}\cdot u_{0}=1$\\
It follows $S_{1}$ is an orthonormal set and $S_{1}\times S_{1}=\pm S_{1}$. 
\end{proof}
\begin{lemma}
$S_{k}$ is an orthonormal set. Moreover, $S_{k}\times S_{k}=\pm S_{k}$ and $\left|S_{k}\right|=2^{k+1}-1$ .
\end{lemma} 
\begin{proof}
We proceed by induction. When $k=1$ claim follows from previous Lemma. Suppose  $S_{k-1}$ is an orthonormal set, $S_{k-1}\times S_{k-1}=\pm S_{k-1}$, and $\left|S_{k-1}\right|=2^{k}-1$. Let $y_{1}, y_{2}\in S_{k-1}$. By definition any element of $S_{k}$ is of the form $y_{1}$, $u_{k}$, or $y_{1} \times u_{k}$. By (1.1), (1.2), and (1.6) we have:\\
(1) $y_{1}\cdot(y_{2} \times u_{k})=u_{k}\cdot(y_{1}\times y_{2})=0$ since $u_{k}\perp S_{k-1}$ and $y_{1}\times y_{2}\in S_{k-1}$.\\
(2) $(y_{1} \times u_{k})\cdot (y_{2}\times u_{k})=u_{k}\cdot((u_{k}\times y_{1})\times y_{2})=-u_{k}\cdot (u_{k}\times(y_{1}\times y_{2}))=0$\\
It follows $S_{k}$ is an orthogonal set.\\
Next by (1.1), (1.2), (1.5), and (1.6) we have:\\
(1) $(y_{1} \times u_{k})\times (y_{2}\times u_{k})=-y_{1}\times (u_{k}\times(y_{2}\times u_{k}))=-y_{1}\times y_{2}$\\
(2) $y_{1}\times (y_{2} \times u_{k})= -((y_{1}\times y_{2})\times u_{k})$ and $y_{1}\times(y_{1} \times u_{k})=-u_{k}$\\
(3) $(y_{1} \times u_{k})\cdot(y_{1} \times u_{k})=y_{1}\cdot(u_{k}\times (y_{1}\times u_{k}))=y_{1}\cdot y_{1}=1$\\
It follows $S_{k}$ is an orthonormal set and $S_{k}\times S_{k}=\pm S_{k}$. Lastly, note  $\left|S_{k}\right|=2\left|S_{k-1}\right|+1=2^{k+1}-1$
\end{proof}
Lemma 3 tells us how to construct a multiplication table for the cross product on $S_{k}$. 
 Let's take a look at the multiplication table in $\textbf{R}^{3}$. In order to simplify the notation for the basis elements we will make the following assignments:
 $e_{1}:=u_{0}$, $e_{2}:=u_{1}$, $e_{3}:=u_{0}\times u_{1}$\\
 \\
\begin{tabular} {| c | c |c|c|}
\hline
 $\times$&$e_{1}$&$e_{2}$&$e_{3}$ \\ \hline
 $e_{1}$& 0 & $e_{3}$&$-e_{2}$ \\ \hline
 $e_{2}$& $-e_{3}$&0 &$e_{1}$ \\ \hline
 $e_{3}$&$e_{2}$&$-e_{1}$&0\\ \hline
 \end{tabular}\\
 \\
 In the table above we used (1.2), (1.5) and (1.6) to compute each product. In particular, $e_{1}\times e_{3}=u_{0}\times (u_{0}\times u_{1})=-u_{1}=-e_{2}$. The other products are computed in a similar way.\\
 Now when  $v=x_{1}e_{1}+x_{2}e_{2}+x_{3}e_{3}$ and $w=y_{1}e_{1}+y_{2}e_{2}+y_{3}e_{3}$ we have $v\times w=(x_{1}e_{1}+x_{2}e_{2}+x_{3}e_{3})\times(y_{1}e_{1}+y_{2}e_{2}+y_{3}e_{3})=
 x_{1}y_{1}(e_{1}\times e_{1})+x_{1}y_{2}(e_{1}\times e_{2})+x_{1}y_{3}(e_{1}\times e_{3})+x_{2}y_{1}(e_{2}\times e_{1})+x_{2}y_{2}(e_{2}\times e_{2})+x_{2}y_{3}(e_{2}\times e_{3})+x_{3}y_{1}(e_{3}\times e_{1})+x_{3}y_{2}(e_{3}\times e_{2})+x_{3}y_{3}(e_{3}\times e_{3})$. By the multiplication table above we can simplify this expression. Hence,\\
$v\times w=(x_{2}y_{3}-x_{3}y_{2})e_{1}+(x_{3}y_{1}-x_{1}y_{3})e_{2}+(x_{1}y_{2}-x_{2}y_{1})e_{3}$\\
The reader should note that this cross product is the standard right-hand-rule cross product.
  
Let's next take a look at the multiplication table in $\textbf{R}^{7}$.
  In order to simplify the notation for the basis elements we will make the following assignments:
 $e_{1}:=u_{0}$, $e_{2}:=u_{1}$, $e_{3}:=u_{0}\times u_{1}$, $e_{4}:=u_{2}$, $e_{5}:= u_{0}\times u_{2}$, $e_{6}:=u_{1}\times u_{2}$, $e_{7}:=(u_{0}\times u_{1})\times u_{2}$.\\
 \\
\\
\begin{tabular} {| c | c |c|c|c|c|c|c|}
\hline
  $\times$&$e_{1}$&$e_{2}$&$e_{3}$&$e_{4}$&$e_{5}$&$e_{6}$&$e_{7}$ \\ \hline
  $e_{1}$& 0 & $e_{3}$&$-e_{2}$&$e_{5}$&$-e_{4}$&$-e_{7}$&$e_{6}$ \\ \hline
  $e_{2}$& $-e_{3}$&0 &$e_{1}$ &$e_{6}$&$e_{7}$&$-e_{4}$&$-e_{5}$\\ \hline
  $e_{3}$&$e_{2}$&$-e_{1}$&0&$e_{7}$&$-e_{6}$&$e_{5}$&$-e_{4}$\\ \hline
  $e_{4}$&$-e_{5}$&$-e_{6}$&$-e_{7}$&0&$e_{1}$&$e_{2}$&$e_{3}$\\ \hline
  $e_{5}$&$e_{4}$&$-e_{7}$&$e_{6}$&$-e_{1}$&0&$-e_{3}$&$e_{2}$\\ \hline
  $e_{6}$&$e_{7}$&$e_{4}$&$-e_{5}$&$-e_{2}$&$e_{3}$&0&$-e_{1}$\\ \hline
  $e_{7}$&$-e_{6}$&$e_{5}$&$e_{4}$&$-e_{3}$&$-e_{2}$&$e_{1}$&0\\ \hline
\end{tabular}\\
\\
In the table above we used (1.2), (1.5) and (1.6) to compute each product. In particular, $e_{5}\times e_{6}=(u_{0}\times u_{2})\times(u_{1}\times u_{2})=-(u_{0}\times u_{2})\times(u_{2}\times u_{1})=((u_{0}\times u_{2})\times u_{2})\times u_{1}=-(u_{2}\times(u_{0}\times u_{2}))\times u_{1}=-u_{0}\times u_{1}=-e_{3}$. The other products are computed in a similar way.\\
\\
Let's look at $v\times w$ when $v=(x_{1},x_{2},x_{3},x_{4},x_{5},x_{6},x_{7})$ and $w=(y_{1},y_{2},y_{3},y_{4},y_{5},y_{6},y_{7})$. Using the bilinear property of the cross product and the multiplication table for $\textbf{R}^{7}$ we have:\\
 \\
 $v\times w=(-x_{3}y_{2}+x_{2}y_{3}-x_{5}y_{4}+x_{4}y_{5}-x_{6}y_{7}+x_{7}y_{6})e_{1}+(-x_{1}y_{3}+x_{3}y_{1}-x_{6}y_{4}+x_{4}y_{6}-x_{7}y_{5}+x_{5}y_{7})e_{2}+(-x_{2}y_{1}+x_{1}y_{2}-x_{7}y_{4}+x_{4}y_{7}-x_{5}y_{6}+x_{6}y_{5})e_{3}+(-x_{1}y_{5}+x_{5}y_{1}-x_{2}y_{6}+x_{6}y_{2}-x_{3}y_{7}+x_{7}y_{3})e_{4}+(-x_{4}y_{1}+x_{1}y_{4}-x_{2}y_{7}+x_{7}y_{2}-x_{6}y_{3}+x_{3}y_{6})e_{5}+(-x_{7}y_{1}+x_{1}y_{7}-x_{4}y_{2}+x_{2}y_{4}-x_{3}y_{5}+x_{5}y_{3})e_{6}+(-x_{5}y_{2}+x_{2}y_{5}-x_{4}y_{3}+x_{3}y_{4}-x_{1}y_{6}+x_{6}y_{1})e_{7}$
  \begin{lemma}
If $u=(u_{0}\times u_{1})+(u_{1}\times u_{3})$ and $v=(u_{1}\times u_{2})-(((u_{0}\times u_{1})\times u_{2})\times u_{3})$ then $u\times v=0$ and $u \perp v$.  
\end{lemma}
\begin{proof}
Using (1.2), (1.5), (1.6), and the bilinear property it can be shown that $u\times v=0$.
Next, note that $u_{0}\times u_{1}$, $u_{1}\times u_{3}$, $u_{1}\times u_{2}$, and $((u_{0}\times u_{1})\times u_{2})\times u_{3}$ are all elements of $S_{i}$ when $i\geq 3$. By Lemma 3 all these vectors form an orthonormal set. This implies $u \perp v$.
\end{proof} 
\begin{lemma}
If $u=(u_{0}\times u_{1})+(u_{1}\times u_{3})$ and $v=(u_{1}\times u_{2})-(((u_{0}\times u_{1})\times u_{2})\times u_{3})$  then $(u\cdot u)(v\cdot v)\neq (u\times v)\cdot (u\times v)+(u\cdot v)^{2}$\\ 
\end{lemma}
\begin{proof}
Using the previous Lemma and Lemma 3, we have $(u\times v)=0$, $(u\cdot u)=2$, $(v\cdot v)=2$, and $(u\cdot v)=0$. Hence, claim follows.
\end{proof}
 \begin{theorem}
A cross product can exist in $\textbf{R}^{n}$ if and only if n=0, 1, 3 or 7. Moreover, there exists orthonormal bases $S_{1}$ and $S_{2}$ for $\textbf{R}^{3}$ and $\textbf{R}^{7}$ respectively such that  $S_{i}\times S_{i}=\pm S_{i}$ for $i=1$ or $2$.
\end{theorem}
\begin{proof}
By Lemma 3 we have that a cross product can only exist in $\textbf{R}^{n}$ if $n=2^{k+1}-1$. Moreover, Lemmas 4 and 5 tells us that if we try to define a cross product in $\textbf{R}^{2^{k+1}-1}$ then the Pythagorean property does not always hold when $k\geq 3$. It follows $\textbf{R}^{0}$, $\textbf{R}^{1}$, $\textbf{R}^{3}$, and $\textbf{R}^{7}$ are the only spaces on which a cross product can exist. It is left as an exercise to show that the cross product that we defined above for $\textbf{R}^{7}$ satisfies all the properties of a cross product and the zero map defines a valid cross product on $\textbf{R}^{0}$, and $\textbf{R}^{1}$. Lastly, Lemma 3 tells us how to generate orthonormal bases $S_{1}$ and $S_{2}$ for $\textbf{R}^{3}$ and $\textbf{R}^{7}$ respectively such that  $S_{i}\times S_{i}=\pm S_{i}$ for $i=1$ or $2$.
\end{proof}
\section*{acknowledgments}
I would like to thank Dr. Daniel Shapiro, Dr. Alberto Elduque, and Dr. John Baez for their kindness and encouragement and for reading over and commenting on various drafts of this paper. I would also like to thank the editor and all the anonymous referees for reading over and commenting on various versions of this paper. This in turn has lead to overall improvements in the mathematical integrity, readability, and presentation of the paper.    
 
\end{document}